\documentclass[12pt]{article}

\usepackage{datetime}

\usepackage[all]{xy}

\usepackage{amssymb, amsmath, amsthm,eucal}

\usepackage{mathptmx}
\usepackage[scaled=.90]{helvet}
\usepackage{courier}

\usepackage{graphicx}

\usepackage[pdftex]{hyperref}

\hypersetup{
  colorlinks = true,
  linkcolor = black,
  urlcolor = blue,
  citecolor = black,
}

\newtheorem{theorem}{Theorem}[section]
\newtheorem{theorem*}{Theorem}
\newtheorem{proposition}[theorem]{Proposition}
\newtheorem{corollary}[theorem]{Corollary}

\theoremstyle{definition}

\renewcommand{\d}{\operatorname{d}} 
\renewcommand{\H}{\operatorname{H}}

\newcommand{\ZZ}{\mathbb{Z}} 
\newcommand{\RR}{\mathbb{R}} 
\newcommand{\CC}{\mathbb{C}}

\newcommand{\TT}{\mathbb{T}}
\newcommand{\GG}{\mathbb{G}}

\newcommand{\Cok}{\operatorname{Cok}}
\newcommand{\Ker}{\operatorname{Ker}}

\newcommand{\Hom}{\operatorname{Hom}}

\newcommand{\B}{\operatorname{B}\!}

\newcommand{\Aut}{\operatorname{Aut}}
\newcommand{\BAut}{\B\Aut}
\newcommand{\Diff}{\operatorname{Diff}}
\newcommand{\BDiff}{\B\Diff}
\newcommand{\Spiff}{\operatorname{Spiff}}
\newcommand{\BSpiff}{\B\Spiff}
\newcommand{\Spiffc}{\operatorname{Spiff}^c}
\newcommand{\BSpiffc}{\B\Spiffc}

\newcommand{\Spinc}{\operatorname{Spin}^c}

\title{\bf Characteristic cohomotopy classes \hbox{ for families
    of 4-manifolds}}

\author{Markus Szymik}

\date{August 2008}

\begin{document}

\maketitle

\begin{abstract}
\noindent
  Families of smooth closed oriented 4-manifolds with a
  complex spin structure are studied by means of a family
  version of the Bauer--Furuta invariants in the context of
  parametrised stable homotopy theory, leading to a
  definition of characteristic cohomotopy classes on Thom
  spectra associated to the classifying spaces of their
  complex spin diffeomorphism groups. This is illustrated
  with mapping tori of such diffeomorphisms and related to
  the equivariant invariants.
\end{abstract}

\thispagestyle{empty}


\section*{Introduction}

The suggestion to extend the gauge theoretical invariants of
smooth 4-manifolds to families and diffeomorphism groups of
such has been around for quite a while,
see~\cite{Donaldson:SW}, \cite{Friedman},
or~\cite{Bauer:Survey}, for example. And there have already
been some efforts in this direction, see
\cite{Ruberman:Obstruction}, \cite{Ruberman:Polynomial},
\cite{Ruberman:PSC}, \cite{LiLiu}, and \cite{Nakamura},
mostly in the context of Seiberg-Witten invariants. This
paper addresses the issue in the context of the Bauer--Furuta
invariants~\cite{BauerFuruta}, bringing affairs to a state
which is conceptionally pleasing and accessible for
calculations.

All 4-manifolds considered here will be closed and
oriented. For simplicity it will also be assumed that the
first Betti number vanishes. It is shown that there are
natural characteristic cohomotopy classes for families of
complex spin 4-manifolds, generalising the invariants of
Bauer and Furuta~\cite{BauerFuruta}.  As the latter are
stable cohomotopy classes, the natural context for family
invariants seems to be fibrewise stable homotopy theory. See
\cite{CrabbJames} for an elementary approach and
\cite{MaySigurdsson} for a more technical treatment.
Section~\ref{sec:familyinvariants} explains how the monopole
map of a family may be used to construct stable homotopy
classes over the base of the family,
following~\cite{BauerFuruta}, see
Theorem~\ref{thm:existence}.

Section~\ref{sec:functoriality} discusses an important
property of the family invariants: functoriality under
pullbacks, see~\ref{thm:pullbacks}. Recall that
characteristic classes for vector bundles can be described
in two ways. Either by giving a universal class in the
cohomology of the classifying space. Or by assigning to
every vector bundle a class in the cohomology of the base
such that these classes are compatible under
pullbacks. These two descriptions are equivalent up
to~$\lim^1$-terms. The former seems to be more suitable in
the statements of the results, while the latter is used in
the proofs. This is how we will proceed in this
paper. Functoriality under pullbacks implies that for every
complex spin 4-manifold~$(X,\sigma_X)$ there is a universal
characteristic class over classifying space of the complex
spin diffeomorphism group~$\BSpiff^c(X,\sigma_X)$,
see~\ref{thm:universal}.
Section~\ref{sec:classifyingspaces} contains a detailed
description of what these spaces classify and how they
relate to the ordinary diffeomorphism groups. Functoriality
under pullbacks also determines the invariants of product
families, see Corollary~\ref{cor:products}.

The family invariants are non-trivial for trivial reason,
much the same as with Thom classes for vector bundles: the
restriction of a family to a point in the base yields the
Bauer--Furuta invariant of the fibre over that point,
see~\ref{cor:fibres}. Thus, the classes constructed here
behave rather like Thom classes than like Euler classes. In
fact, there is an interpretation of the family invariants
using the ordinary stable cohomotopy of Thom spectra for
families of Fredholm operators. See Theorem~\ref{thm:thom}
in Section~\ref{sec:thomspectra}.

Section~\ref{sec:K3} applies this to gain information about
families of K3 surfaces. The example of the universal Kummer
family also shows that the family invariant need not be
determined by the ordinary invariant of the
fibre. Section~\ref{sec:mapping_tori} illustrates how the
characteristic cohomotopy classes for mapping tori may be
used to define an invariant of isotopy classes of
diffeomorphisms of 4-manifolds. The final
Section~\ref{sec:diss} comments on the relationship of the
family invariants with the equivariant invariants
of~\cite{Szymik:Diss}.


\section{Classifying spaces for spin diffeomorphism\texorpdfstring{\\}{} 
  groups}\label{sec:classifyingspaces}

If~$X$ is a closed oriented manifold, the
notation~$\Diff(X)$ will refer to the group of
diffeomorphisms of~$X$ which preserve the orientation. The
space~$\BDiff(X)$ classifies smooth families~$p:Y\rightarrow
B$ with an orientation on the relative tangent bundle such
that the fibres are diffeomorphic to~$X$ as oriented
manifolds. One may just as well assume that there is a
metric on the relative tangent bundle since there is a
contractible choice of these. One gets something different
if one requires the fibres to be isometric to~$X$ with a
given metric; this defines the classifying space of the
isometry group of that metric.

If~$\sigma_X$ is a (real) spin structure on~$X$, instead of
looking at the group of diffeomorphisms~$f$ which preserve
the spin structure in the sense that~$f^*\sigma_X$ is
isomorphic to~$\sigma_X$, one may consider the
group~$\Spiff(X,\sigma_X)$ of pairs~$(f,u)$, where~$f$ is an
orientation preserving diffeomorphism of~$X$
and~\hbox{$u:f^*\sigma_X\rightarrow\sigma_X$} is an
isomorphism of spin structures. The space~$\BSpiff(X)$
classifies smooth families~$p:Y\rightarrow B$ with a spin
structure on the relative tangent bundle such that the
fibres are diffeomorphic to~$X$ as spin manifolds. By
definition, there is an exact sequence
\begin{displaymath}
  1\longrightarrow
  \Aut(\sigma_X)\longrightarrow
  \Spiff(X,\sigma_X)\longrightarrow
  \Diff(X).
\end{displaymath}
The 2-torus shows that the rightmost arrow need not be
surjective. But since there are only finitely many spin
structures on~$X$, the classifying space of the image is a
finite covering of~$\BDiff(X)$.  The group~$\Aut(\sigma_X)$
is~$\ZZ/2$.
The exact sequence leads to a fibration between the
classifying spaces. The~2-sphere shows that this need not
split. However, this fibration shows that the
map~$\BSpiff(X,\sigma_X)\rightarrow\BDiff(X)$ induces
isomorphisms between the higher homotopy groups~$\pi_n$ for
$n\geqslant3$.

A group~$\Spiffc(X,\sigma_X)$ can be defined similarly,
using a complex spin structure~$\sigma_X$ on~$X$. Note that
in the description of the families classified
by~$\BSpiffc(X,\sigma_X)$ one may require a unitary
connection on the determinant line bundle since there is a
contractible choice of these. However we do not require that
the restrictions to the fibres give a fixed connection
for~$\sigma_X$.  There is an exact sequence and a fibration
as above. This time the group~$\Aut(\sigma_X)$ is the
(gauge) group of maps from~$X$ to~$\TT$.  Since~$b^1(X)$ is
assumed to vanish, the space~$\BAut(\sigma_X)$ is equivalent
to a copy of~$\B\TT$. Note that this differs from the
description from the previous paragraph in case the complex
spin structure~$\sigma_X$ comes from real spin
structure. The fibration
\begin{displaymath}
  \B\TT
  \longrightarrow
  \BSpiffc(X,\sigma_X)\longrightarrow
  \BDiff(X)
\end{displaymath}
shows that the
map~\hbox{$\BSpiffc(X,\sigma_X)\rightarrow\BDiff(X)$}
induces isomorphisms between the higher homotopy
groups~$\pi_n$ for~$n\geqslant4$.


\section{Bauer--Furuta invariants for families}\label{sec:familyinvariants}

This section contains a framework to extend the Bauer--Furuta
invariants~\cite{BauerFuruta} to families of
smooth~4-manifolds. First an appropriate notion of a family
of~4-manifolds with complex spin structure and the
corresponding monopole map are explained. Then the
invariants are defined.

\subsection{The ingredients}\label{subsec:ingredients}

A smooth family~$Y|B$ of 4-manifolds is given by a smooth
submersion~\hbox{$p:Y\rightarrow B$} between smooth manifolds~$Y$
and~$B$, such that the relative dimension of~$p$
is~\hbox{$\dim(Y)-\dim(B)=4$}. For any point~$b$ in~$B$, the
fibre of~$p$ over~$b$ is~a~4-manifold~$Y(b)$. All fibres are
diffeo\-morphic if the base~$B$ is connected, which will be
assumed throughout this section. The differential
\begin{equation}\label{differential_of_submersion}
  T_Y\longrightarrow p^*T_B
\end{equation}
is surjective. Its kernel~$T_{Y|B}$ is a 4-dimensional vector bundle
over~$Y$, a subbundle of~$T_Y$.
This is the relative tangent bundle, or vertical tangent
bundle, or the tangent bundle along the fibres. The fibre in
a point~$y$ is the tangent bundle in~$y$ of the
fibre~$Y(p(y))$ which passes through~$z$.  Similarly, the
relative cotangent bundle~$T_{Y|B}^{\scriptscriptstyle\vee}$
is the cokernel of the injection dual
to~\eqref{differential_of_submersion}. We will only consider
oriented families, i.e. families such that the
bundles~$T_{Y|B}$ and~$T_{Y|B}^{\scriptscriptstyle\vee}$ are
oriented.  The orientations induce orientations on every
fibre.  (The Klein bottle shows that families are not
necessarily orientable just because the fibres and the base
are.)
As explained in Section~\ref{sec:classifyingspaces}, it can
and will be assumed that there is a metric on the
bundle~$T_{Y|B}$.

In order to define the monopole map for a smooth
family~$Y|B$, a family of complex spin structures is
required, i.e. a complex spin structure~$\sigma_{Y|B}$ on
the relative tangent bundle. Note that this is a
4-dimensional vector bundle over~$Y$, so that a complex spin
structure means that the structure group of that bundle is
reduced to the group~$\Spinc(4)$. The spinor bundles will
usually be denoted by~$W^{\,\pm}(\sigma_{Y|B})$. As
explained in Section~\ref{sec:classifyingspaces}, it can and
will be assumed that a unitary connection~$A$~-- referred to
as the {\it base connection}~-- on the determinant line
bundle~$L(\sigma_{Y|B})$ is fixed.

One may ask the question whether or not one may always find
a complex spin structure on the relative tangent bundle
which restricts to the given one on the fibres. The homotopy
theory from Section~\ref{sec:classifyingspaces} displays a
single obstruction living in~$\H^3(B;\ZZ)$.

\subsection{The monopole map} 

Given a family~$p:Y\rightarrow B$ as above, let
\begin{displaymath}
  \Lambda^k_{Y|B}=
  \Lambda^kT_{Y|B}^{\scriptscriptstyle\vee}\otimes\underline{L\TT}_Y
\end{displaymath}
be the bundle over~$Y$ whose sections are the
alternating~$k$-linear forms on the relative tangent bundle
with values in the Lie algebra~$L\TT$ of~$\TT$. (Here and in
the following, the notation~$\underline{L\TT}_Y$ is used for
the trivial bundle with fibre~$L\TT$ over~$Y$.) The
bundle~$\Lambda^2_{Y|B}$ decomposes as
$\Lambda^+_{Y|B}\oplus\Lambda^-_{Y|B}$ into the self-dual
and the anti-self-dual part. Since the
bundle~$\Lambda^0_{Y|B}$ is the trivial line bundle, the
sections thereof are just the~$L\TT$-valued functions
on~$Y$. Thus, the direct image~$p_*\Lambda^0_{Y|B}$ on~$B$
contains a trivial line bundle isomorphic
to~$\underline{L\TT}_B$ which is generated by the constants,
i.e. by the functions which live on~$B$.

If, in addition, a complex spin structure~$\sigma_{Y|B}$
on~$Y|B$ is given, consider the
(infinite-dimensional)~$\TT$-vector bundles
\begin{eqnarray*}
  \mathcal U &=& p_*W_{Y|B}^+(\sigma_{Y|B})\oplus
  p_*\Lambda^1_{Y|B}\oplus\underline{L\TT}_B\\ \mathcal V
  &=& p_*W_{Y|B}^-(\sigma_{Y|B})\oplus
  p_*\Lambda^+_{Y|B}\oplus p_*\Lambda^0_{Y|B}
\end{eqnarray*}
over the trivial~$\TT$-space~$B$. The \textit{monopole map}
$\mu(Y|B,\sigma_{Y|B})$, defined as
\begin{displaymath}
  (\phi,a,f)\longmapsto(D_{A+a}(\phi),F^+_{A+a}-\phi^2,\d^*(a)-f),
\end{displaymath}
is a~$\TT$-equi\-variant map from~$\mathcal U$ to
$\mathcal V$ over~$B$.

\subsection{The linearisation and its index bundles}  

The linearisation (at the origin) of the monopole map above is given as
the~$\TT$-linear map~$\lambda(Y|B,\sigma_{Y|B})$
from~$\mathcal U$ to~$\mathcal V$ over~$B$ defined by
\begin{equation}\label{linearisation}
  (\phi,a,f)\longmapsto(D_A(\phi),\d^+(a),\d^*(a)-f).
\end{equation}
It is fibrewise Fredholm over~$B$, so that it has an
index. That index is a virtual vector bundle over~$B$,
i.e. an element in the~$K$-theory of~$B$, in~$KO_\TT(B)$ to
be precise.~(Since~$B$ is a trivial~$\TT$-space, the
group~$KO_\TT(B)$ splits into a copy of~$KO(B)$ and copies
of~$KU(B)$, one for each positive integer.) As one sees
from~\eqref{linearisation}, the index of the linearisation
is the sum of the index of the family of Dirac operators and
the index of the family of fundamental elliptic complexes.

The family of complex linear Dirac operators defines a class
in the group~$KU(B)$, which will be identified with its
image in~$KO_\TT(B)$.  It seems harmless to call this class
the {\it Dirac bundle} for short.  Its rank
is~\hbox{$(c_1^2-s)/8$}, where~$s$ is the signature of the
fibres, and~$c_1$ refers to the first Chern class of the
determinant line bundle restricted to the fibres.

The other summand is the index of a family of a real
opera\-tors, so that the index is a class of rank~$b^+(X)$
in the subgroup~$KO(B)$ of~$KO_\TT(B)$. This class is in
fact the additive inverse of the {\it plus bundle}, namely
of the flat vector bundle which is the self-dual
subbundle~\hbox{$R^+p_*\underline\RR_Y$}
of~$R^2p_*\underline\RR_Y$.  (Here, the
symbol~$\underline\RR_Y$ denotes the sheaf of locally
constant functions with values in~$\RR$ with the discrete
topology.) See~\cite{Atiyah:Signature}
and~\cite{AtiyahSinger}.
If the base space~$B$ of the family is simply-connected,
this flat bundle will always be \hbox{trivial}.
 
\subsection{The invariants for families} 

Let~$B$ be compact base manifold, and
let~$(Y|B,\sigma_{Y|B})$ be a complex spin family over~$B$
as in \ref{subsec:ingredients}. The monopole
map~$\mu(Y|B,\sigma_{Y|B})$ is a continuous map between
Hilbert space bundles which is the sum of its linearisation
$\lambda=\lambda(Y|B,\sigma_{Y|B})$, which is fibrewise
Fredholm over~$B$, and another map~$\kappa$. The arguments
in~\cite{BauerFuruta} show that~$\kappa$ is a
fibre-preserving map which sends bounded disk bundles to
subspaces which are proper over~$B$. Furthermore pre-images
under~$\mu(Y|B,\sigma_{Y|B})$ of bounded disk bundles are
contained in bounded disk bundles.

Let us review the construction from~\cite{BauerFuruta} in
this context. One chooses~a finite-dimensional subbundle~$V$
inside~$\mathcal{V}$. It is required that~$V$ is
sufficiently large, so that it contains the subspaces
$\Cok(\lambda)$. Tautologically,~$\lambda$ maps the
subbundle~$\lambda^{-1}(V)$ of~$\mathcal{U}$ , which
contains~$\Ker(\lambda)$, into~$V$. Essentially by
assumption,~$\mu$ maps~$\lambda^{-1}(V)$ near to~$V$. In
particular, after possibly enlarging~$V$ further, one can
achieve that the sphere bundle~$S_B(\mathcal{V}-V)$ is
missed. Then the fibrewise one-point-compactification~$S_B^{\lambda^{-1}(V)}$ is mapped into
$S_B^{\mathcal{V}}\backslash S_B(\mathcal{V}-V)$, which can
be retracted to~$S_B^V$. This gives a map
from~$S_B^{\lambda^{-1}(V)}$ to~$S_B^V$. While this is
morally the map one wants, taken as it is, it does not
define~a class
in~$[S_B^{\Ker(\lambda)},S_B^{\Cok(\lambda)}]^\TT_B$ for the
universe~$\mathcal{V}$. However,~\hbox{$V=\Cok(\lambda)\oplus(V-\Cok(\lambda))$},
and the map~$\lambda$ induces an isomorphism of
$\lambda^{-1}(V)-\Ker(\lambda)$ with~$V-\Cok(\lambda)$.
Thus, there is~a unique dashed arrow such that the diagram
\begin{center}
  \mbox{
    \xymatrix@R=20pt{
      S_B^{\Ker(\lambda)\oplus(V-\Cok(\lambda))}\ar@{-->}[r]
      & S_B^{\Cok(\lambda) \oplus
        (V-\Cok(\lambda))}\\
      S_B^{\Ker(\lambda)\oplus
        (\lambda^{-1}(V)-\Ker(\lambda))}\ar[u]^-{\cong}  & \\
      S_B^{\lambda^{-1}(V)}\ar@{=}[u]\ar[r] & S_B^V\ar@{=}[uu] 
  }}
\end{center}
commutes. It represents an element in
$[S_B^{\Ker(\lambda)},S_B^{\Cok(\lambda)}]^\TT_B$ which does
not depend on~$V$.

\begin{theorem}\label{thm:existence}
  For every complex spin family~$(Y|B,\sigma_{Y|B})$
  over~$B$, the monopole map defines a class in the
  group~$[S^{\Ker(\lambda)}_B,S^{\Cok(\lambda)}_B]^\TT_B$.
\end{theorem}

Standard arguments -- using the contractibility of the
choices involved -- show that this class neither depends on
the metric nor on the connection.

This is the Bauer--Furuta invariant of~$Y|B$ with respect to
$\sigma_{Y|B}$. As with any other invariant, the
computability of these family invariants depends on
structure theorems which describe how the invariants change
when the families are changed. The following section
contains such a structure theorem for the family invariants.


\section{Functoriality}\label{sec:functoriality}

In this section, a complex spin 4-manifold~$(X,\sigma_X)$ is
fixed, and complex spin families with fibre~$(X,\sigma_X)$
are considered for varying base~$B$.
Given such a family over~$B$, a morphism~$B'\rightarrow B$ induces a
pullback family over~$B'$. Pullback also induce a base change functor
from the stable homotopy category over~$B$ to the stable homotopy
category over~$B'$, see~\cite{CrabbJames}. Inspection of the
definitions immediately given the following structure result.

\parbox{\linewidth}{\begin{theorem}\label{thm:pullbacks}
  Given a family over~$B$ and a map~$B'\rightarrow B$, the invariant
  of the pullback family over~$B'$ is the image of the invariant of
  the family over~$B$ under the homomorphism
  \begin{displaymath}
    [S^{\Ker(\lambda)}_B,S^{\Cok(\lambda)}_B]^\TT_B
    \longrightarrow
    [S^{\Ker(\lambda')}_{B'},S^{\Cok(\lambda')}_{B'}]^\TT_{B'},
  \end{displaymath}
  induced by the base change functor.
\end{theorem}}

Let us see what this means for two distinguished classes of examples:
if~$B'$ or~$B$ is a point.

\subsection{Fibres} 

A morphism from a single point~$b$ into~$B$ as above corresponds to an
identification of~$(X,\sigma_X)$ with the fibre over~$b$.

\parbox{\linewidth}{\begin{corollary}\label{cor:fibres}
  The homomorphism
  \begin{displaymath}
    [S^{\Ker(\lambda)}_B,S^{\Cok(\lambda)}_B]^\TT_B
    \longrightarrow
    [S^{\Ker(\lambda)},S^{\Cok(\lambda)}]^\TT
  \end{displaymath}
  corresponding to the inclusion of a point in~$B$ sends the family
  invariant of a family over~$B$ to the ordinary invariant of its
  fibre~$(X,\sigma_X)$ over that point.
\end{corollary}}

This shows that the ordinary Bauer--Furuta
invariants are contained in the family Bauer--Furuta invariants.

\subsection{Products} 

A morphism from~$B$ to a single point corresponds to the product
family~$B\times X$ with the product complex spin
structure~$B\times\sigma_X$ over~$B$. All data are pullbacks from~$X$,
considered as a family over a point. Note that in this case the plus
bundle is the trivial bundle, and similarly for the Dirac bundle.

\begin{corollary}\label{cor:products}
  The invariant of a product family~$(B\times X, B\times\sigma_X)$ is
  the product of the ordinary invariant of~$(X,\sigma_X)$ with the
  identity of~$B$.
\end{corollary}

Therefore the family invariants in product situations are completely
understood in terms of the ordinary invariants.

\subsection{A universal characteristic class}

Fibrewise stable homotopy theory works well only if the base
$B$ is homotopy equivalent to a compact ENR or a finite
CW-complex. In order to handle the universal base
$\BSpiffc(X,\sigma_X)$ for the situation at hand, on better
defines
\begin{displaymath}
  [?,??]^\TT_{\BSpiffc(X,\sigma_X)}\;=\;\lim\limits_B\;[?,??]^\TT_B,
\end{displaymath}
where the limit is over the subspaces~$B$ of
$\BSpiffc(X,\sigma_X)$ which satisfy the finiteness
hypothesis and the corresponding restriction maps. See
Section~II.15 in~\cite{CrabbJames}.
Theorem~\ref{thm:pullbacks} then immediately gives the
following result.

\begin{theorem}\label{thm:universal}
  If~$(X,\sigma_X)$ is a complex spin 4-manifold, the group
  \begin{displaymath}
    [S^{\Ker(\lambda)}_B,S^{\Cok(\lambda)}_B]_{\BSpiffc(X,\sigma_X)}^\TT
  \end{displaymath}
  contains a universal characteristic class for families of
  complex spin 4-manifolds with fibre~$(X,\sigma_X)$.
\end{theorem}

Note that this class is non-trivial for trivial reasons in
general: restricted to a point it gives back the ordinary
Bauer--Furuta invariant of~$(X,\sigma_X)$. In this respect,
this characteristic class behaves more like a Thom class
than like an Euler class. In fact, an interpretation in
terms of Thom spectra will be given in the next section.


\section{Cohomotopy classes of Thom spectra}\label{sec:thomspectra}

In this section, we will see how the fibrewise stable
homotopy theory of the previous two sections, which is
conceptionally preferable, can be translated~-- using one of
the basic adjunctions of the game -- into ordinary stable
cohomotopy of Thom spectra, which might be more accessible
for calculations. The motivation for the use of spectra is
the same as in the ordinary case: while the sphere in a
Hilbert space is contractible, it can be approximated by a
sphere spectrum, which is not. Here, Thom spectra for
families of Fredholm operators will be defined, such that
the sphere spectrum can be recovered as the Thom spectrum of
the identity over a singleton.

\subsection{Thom spectra}

Let~$\mathcal U$ and~$\mathcal V$ be bundles of Hilbert
spaces over a compact base~$B$. If~\hbox{$\lambda:\mathcal
U\rightarrow\mathcal V$} is a family of Fredholm operators
over~$B$, we will now see how to obtain a Thom
spectrum~$B^\lambda$ in that situation.

Let us choose a Hilbert space~$H$, which also serves as the
universe for the spectrum, and a
trivialisation~\hbox{$t:\mathcal V\cong\underline H_B$} over
$B$. This exists by a theorem of
Kuiper's~\cite{Kuiper}. Consider then the composition
\begin{displaymath}
  \lambda_t:\mathcal U\longrightarrow\mathcal V\cong\underline
  H_B\longrightarrow H
\end{displaymath}
with the projection onto~$H$. For every sufficiently large
finite-dimensional subspace~$V$ of~$H$ the
pre-image~$\lambda_t^{-1}V$ is a subbundle of~$\mathcal U$
which is mapped to~$V$ by~$\lambda_t$. Let
$B^{\lambda_t^{-1}V}$ be the Thom space of this bundle; this
is the quotient of the fibrewise one-point-compactification
$S^{\lambda_t^{-1}V}_B$ by the section at infinity. If one
enlarges~$V$ to~$W$ the map~$\lambda_t$ induces an
isomorphism
\begin{displaymath}
  B^{\lambda_t^{-1}W}\cong\Sigma^{W-V}B^{\lambda_t^{-1}V}
\end{displaymath}
of Thom spaces. These define the spaces and structure maps
of the Thom spectrum~$B^{\lambda_t}$ of~$\lambda$ with
respect to the trivialisation~$t$. In order to show
that~$B^{\lambda_t}$ is essentially independent of~$t$, let
us consider another trivialisation~$B^{\lambda_t}$. This
differs from~$t$ by an automorphism of~$\underline H_B$,
i.e. by a map from~$B$ to the unitary group of~$H$, which is
contractible again by Kuiper's theorem. Using such a map one
gets an essentially unique identification of the two Thom
spectra~$B^{\lambda_t}$ and~$B^{\lambda_{t'}}$. This
justifies the notation~$B^\lambda$ to be used later.

As for functoriality, given a morphism of another space~$B'$
into~$B$, one may consider the pullback, say~$\lambda'$, of
a maps~$\lambda$ as above. Then there is a commutative
diagram
\begin{center}
  \mbox{ \xymatrix@R=10pt{ \mathcal U\ar[r]^-{\lambda} &
      \mathcal V\ar[r]^-\cong& \underline H_B\ar[dr] & \\
      &&& H.\\ \mathcal U'\ar[r]_-{\lambda'}\ar[uu]
      &\mathcal V'\ar[r]_-\cong\ar[uu] & \underline
      H_{B'}\ar[uu]\ar[ur] & } }
\end{center}
The bundle~$(\lambda')^{-1}V$ over~$B'$ is the pullback of
the bundle~$\lambda^{-1}V$ over~$B$. Thus there is a map
from~$(B')^{(\lambda')^{-1}V}$ to~$B^{\lambda^{-1}V}$.
These fit together to give a morphism
\begin{equation}
  \label{eq:thom_map}
  (B')^{\lambda'}\longrightarrow B^\lambda 
\end{equation}
of spectra.

Everything works without changes in the case of a compact Lie group
acting on everything in sight. See~\cite{Segal:Kuiper}, for the
extension of Kuiper's theorem to this setting.

\subsection{Cohomotopy classes}

If~$T$ is a space, and~$S$ is a space over~$B$, there is a
natural bijection between the set of maps from~$S$ to~$T$
and the set a fibrewise map from~$S$ to~$B\times T$ over
$B$. In the pointed context, this passes to an adjunction
$[S/B,T]\cong[S,B\times T]_B$ between the direct image
functor from the stable homotopy category over~$B$ to the
ordinary stable homotopy category and the pullback functor
in the other direction. Together with the observation that
the direct image of a sphere spectrum~$S^V_B$ is the Thom
spectrum~$S^V_B/B=B^V$, this yields the following result.

\parbox{\linewidth}{\begin{theorem}\label{thm:thom}
  The monopole map of a family of complex spin~4-manifolds
  over~$B$ with fibre~$(X,\sigma_X)$ defines a stable
  cohomotopy class in the
  group~\hbox{$\pi^0_\TT(B^\lambda)$}. The cohomotopy class
  of a pullback family along a map~$B'\rightarrow B$ is the
  composition with the morphism~{\upshape
  (\ref{eq:thom_map})}: the diagram
  \begin{center}
    \mbox{ \xymatrix@C=50pt@R=0pt{B^\lambda\ar[dr]^{\mu}& \\
      & S^0\\
      (B')^{\lambda'}\ar[uu]^{(\ref{eq:thom_map})}\ar[ur]_{\mu'}&
      }}
  \end{center}
  commutes. Therefore, these classes define a universal
  characteristic class in the
  group~$\pi^0_\TT(\BSpiffc(X,\sigma_X)^\lambda)$.
\end{theorem}}

The group~$\pi^0_\TT(\BSpiffc(X,\sigma_X)^\lambda)$ should
be thought of as the cohomotopy of the classifying
space~$\BSpiffc(X,\sigma_X)$ in degree~$\lambda$.


\section{Families of K3 surfaces}\label{sec:K3}

Just as many of examples of smooth 4-manifolds are given as
complex surfaces, holomorphic families may be quarried for
interesting examples of smooth families. However, finding
non-trivial complete families of complex surfaces is
esteemed to be hard. A classical example can be constructed
from the tautological family of~4-tori over the moduli
space~$SO(4)/U(2)\cong S^2$ of complex structures on~$\RR^4$: the fibrewise Kummer construction yields a
holomorphic family of K3 surfaces over the projective line,
see~\cite{Atiyah:Kummer}, or \cite{Borcherdsetal} for a
recent appearance. More generally, to illustrate the
preceding results, let us consider families of K3 surfaces
over~$S^2$ with the standard (real) spin structure. The
family need not come with a real spin structure on the
relative tangent bundle, but there will always exits a
complex spin structure which restricts to the spin
structures on the fibres. The complex rank of the associated
Dirac bundle is~$-s(K3)/8=2$, and the rank of the plus
bundle is~\hbox{$b^+(K3)=3$}. The bundle of self-dual
harmonic~2-forms is clearly trivial since the base of the
family is simply-connected. However, the Dirac bundle need
not be trivial. In fact, the universal Kummer family above
is distinguished by the fact that its Dirac bundle has degree
-2, see~\cite{Atiyah:Kummer}. In order to say something in
general, let us pause to review some homotopy theory of Thom
spaces.

Let~$V$ be a complex vector bundle of rank~$r$ over a
connected space~$B$. If~$B$ has a~CW-filtration with
quotients which are wedges of spheres~$S^n$, the Thom
space~$B^V$ has a filtration with quotients which are wedges
of~$\TT$-spheres~$S^{r\CC+n}$. This is non-equi\-variantly a~CW-filtration. The bottom cell of the Thom space is given by
the class of a compactified fibre. If the base space~$B$ is
itself an~$n$-sphere, the situation is particularly
easy. The CW-structure with two cells leads to a cofibre
sequence
\begin{equation}\label{cofibre_sequence_over_S^n}
  S^{r\CC+(n-1)}\longrightarrow S^{r\CC}\longrightarrow B^V\longrightarrow
  S^{r\CC+n}\longrightarrow S^{r\CC+1}.
\end{equation}
Therefore, the~$\TT$-equi\-variant stable homotopy type of
the Thom space of~$V$ is described by an element~$J_V$
in~$[S^{r\CC+(n-1)},S^{r\CC}]^\TT\cong\pi^\TT_{n-1}$. The
assignment~$V\mapsto J_V$ is a version of
the~J-homomorphism.

Mapping the cofibre sequence
\eqref{cofibre_sequence_over_S^n} into a~$\TT$-trivial
sphere~$S^b$, one obtains a long exact sequence
\begin{displaymath}
  [S^{r\CC+(n-1)},S^b]^\TT\leftarrow [S^{r\CC},S^b]^\TT\leftarrow
  [B^V,S^b]^\TT\leftarrow [S^{r\CC+n},S^b]^\TT\leftarrow 
  [S^{r\CC+1},S^b]^\TT.
\end{displaymath}
Note that in the particular case~$n=2$ the left most and
right most groups are the same. This is relevant in the
case~$r=2$,~$b=3$, and~$n=2$ encountered for the K3 families
above. The structure of the sequence
\begin{equation}\label{special case}
  [S^{2\CC+1},S^3]^\TT\leftarrow[S^{2\CC},S^3]^\TT\leftarrow
  [B^V,S^3]^\TT\leftarrow[S^{2\CC+2},S^3]^\TT\leftarrow
  [S^{2\CC+1},S^3]^\TT
\end{equation}
is explained by the following result.

\parbox{\linewidth}{\begin{proposition}\label{sample calculation}  
  Let~$V$ be a complex vector bundle of rank 2 over a
  2-sphere~$B$. Consider the map
  \begin{displaymath}
        \ZZ\cong[S^{2\CC},S^3]^\TT \longleftarrow [B^V,S^3]^\TT
  \end{displaymath}
  induced by the inclusion~$S^{2\CC}\subset B^V$ of a fibre.
  The degree of~$V$ is odd if and only if this map is
  injective with a cokernel of order 2. It is even if and
  only if this map is surjective with a kernel of order 2.
\end{proposition}}

\begin{proof}
  Let us consider the exact sequence~\eqref{special
  case}. The two maps on the left and on the right are given
  by multiplication with~$J_V$. To start with, let us look
  at the structure of the groups involved in~\eqref{special
  case}. It is as follows.
  \begin{displaymath}
    {[}S^{2\CC},S^3{]}^\TT\cong\ZZ\hspace{25pt}
    {[}S^{2\CC+1},S^3{]}^\TT\cong\ZZ/2\hspace{25pt}
    {[}S^{2\CC+2},S^3{]}^\TT\cong\ZZ/2
  \end{displaymath}
  The first two isomorphism are easy. One may proceed as in
  \cite{BauerFuruta}. Let us turn to the third one. The
  group~$[S^{2\CC+2},S^3]^\TT\cong[S^{2\CC},S^1]^\TT$ is a
  cokernel of the homomorphism
  \begin{displaymath}
    p^*:[S^0,S^0]^\TT\cong[D(2\CC)_+,S^0]^\TT\longrightarrow[S(2\CC)_+,S^0]^\TT
  \end{displaymath}
  induced by the projection~$p:S(2\CC)_+\subset
  D(2\CC)\simeq_\TT S^0$ which sends~$S(2\CC)$ to the point
  which is not the basepoint. While~$[S^0,S^0]^\TT$ is a
  copy of the integers, generated by the
  identity,~$[S(2\CC)_+,S^0]^\TT$ is isomorphic
  to~$\ZZ\oplus\ZZ/2$, the copy of the integers being
  generated by~$p$. Of course,~$p^*(\mathrm{id})=p$, so
  that~$p^*$ is a split injection, induces an isomorphism
  between the copies of the integers, and has
  cokernel~$\ZZ/2$, as claimed. Up to isomorphism the exact
  sequence~\eqref{special case} thus looks like
  \begin{displaymath}
    \ZZ/2\longleftarrow\ZZ\longleftarrow[B^V,S^3]^\TT
    \longleftarrow\ZZ/2\longleftarrow\ZZ/2.
  \end{displaymath}
  Note that the forgetful
  map~\hbox{$\Phi:[S^{2\CC},S^3]^\TT\rightarrow[S^4,S^3]\cong\ZZ/2$}
  is surjective. By~$\pi_*$-linearity it follows that the
  forgetful
  map~\hbox{$\Phi:[S^{2\CC+1},S^3]^\TT\rightarrow[S^5,S^3]\cong\ZZ/2$}
  is an isomorphism and that the forgetful
  map~\hbox{$\Phi:[S^{2\CC+2},S^3]^\TT\rightarrow[S^6,S^3]\cong\ZZ/24$}
  is injective. It follows that the two maps on the left and
  on the right of~\eqref{special case} are trivial if and
  only if~$\Phi(J_V)$ is zero. They are surjective if and
  only if~$\Phi(J_V)$ is non-zero. However, in the case of a
  2-sphere,~$\Phi(J_V)$ is trivial if and only the
  underlying real bundle of~$V$ is trivial if and only the
  degree of~$V$ is odd.
\end{proof}

In the case of a family of K3 surfaces, the map in
Proposition~\ref{sample calculation} must be surjective: the family
invariant maps to the ordinary Bauer--Furuta invariant of K3, which is
a generator, as the isomorphism with~$\ZZ$ sends it to the
Seiberg-Witten invariant of~$K3$, which is~$\pm1$. This proves the
following result.

\begin{theorem}\label{index for families of K3 surfaces}
  Given a complex spin family of K3 surfaces over~$S^2$ with the
  standard spin structure in each fibre, the index bundle of the Dirac
  operator has even degree.
\end{theorem}

Note that the family invariants are not determined by the ordinary
invariants: each time there are two possible pre-images. It would be
interesting to find families realizing the two different
possibilities.

The universal Kummer family shows that the result is best
possible. Though one is tempted to phrase it in terms of the first
Chern class of the universal Dirac bundle over~$\BSpiff^c(K3)$, this
would require knowledge of the homomorphism
\begin{displaymath}
  \H^2(\BSpiff^c(K3);\ZZ)\longrightarrow\Hom(\pi_2(\BSpiff^c(K3)),\ZZ).
\end{displaymath}
The behaviour of this map depends on the fundamental group
of~$\BSpiff^c(K3)$, which is a problem of its own. The
methods of the following section may be used to gain
information about~$\pi_1(\BSpiff^c(X,\sigma_X))$ in general.


\section{Mapping tori}\label{sec:mapping_tori}

Let~$X$ be a 4-manifold as before. Every diffeomorphism of~$X$ induces
an automorphism of~$\H^2(X)$ which preserves the intersection
form. As~$X$ is simply-connected, the kernel of the resulting
homomorphism from~$\Diff(X)$ to~$\Aut(\H^2(X))$ is known to be the
group of diffeomorphisms pseudo-isotopic to the identity,
see~\cite{Kreck}. However, pseudo-isotopy does not imply isotopy in
dimension~4, see~\cite{Kwasik}. In this section, we show how the
family invariants introduced above may be used to study these
diffeomorphisms.

If~$f$ is a diffeomorphism of~$X$ which is homologically trivial, and
$\sigma_X$ is any complex spin structure on~$X$, there is an
isomorphism~$u$ from~$f^*\sigma_X$ to~$\sigma_X$. The pair~$(f,u)$
gives rise to a~$\ZZ$-action on~$(X,\sigma_X)$. As the discussion of
Section~\ref{sec:classifyingspaces} shows, the surjection from
$\Spiffc(X,\sigma_X)$ to the group of homologically trivial
diffeomorphisms of~$X$ induces an isomorphism on
components. Therefore, the choice of~$u$ does not matter, and need no
longer be subject of our discussion. Let~$Y$ be the mapping torus
of~$f$, the quotient~$X\times_\ZZ\RR$ of~$X\times\RR$ by
the~$\ZZ$-action generated by~$(x,t)\mapsto(f(x),t-1)$. The projection
to the second component displays~$Y$ as the total space of a family
over the circle~$C=\RR/\ZZ$. In the situation given, the index
bundle~$\lambda$ is trivial: the Dirac bundle is trivial as any
virtual complex vector bundle over a circle is trivial, and the plus
bundle is trivial by assumption on~$f$. In fact, the map induced
by~$f$ on~$\H^+(X)$ is the identity, so that the
projection~\hbox{$\H^+(X)\times\RR\rightarrow\H^+(X)\times_\ZZ\RR$}
descends to an
isomorphism~\hbox{$\H^+(X)\times(\RR/\ZZ)\cong\H^+(X)\times_\ZZ\RR$}. As
for the Dirac bundle, the situation is more complicated~-- due to the
lack of control over the action of~$f$ on the space of harmonic
spinors. There does not seem to be a preferred trivialisation, but the
following proposition comes for rescue.

\begin{proposition}
  Let~$A\colon S^1\rightarrow U(r)$ be any map. Then the induced
  self-map of the Thom space~$\Sigma^{r\CC}S^1_+$ of the trivial
  rank~$r$ bundle over~$S^1$ is stably~$\TT$-homotopic to the
  identity.
\end{proposition}

\begin{proof}
  The induced self-map of~$\Sigma^{r\CC}S^1_+$, which is the quotient
  of~$S^1\times S^{r\CC}$ by the section at infinity, is given by the
  formula
  \begin{equation}
    \label{eq:self-map}
    (z,x)\longmapsto(z,A(z)\cdot x).
  \end{equation}
  Stably,~$S^1_+$ splits as~$S^1\vee S^0$, and the self-map respects
  the ensuing decomposition~\hbox{$\Sigma^{r\CC}S^1_+\simeq_\TT
    S^{r\CC+1}\vee S^{r\CC}$} such that it is the identity on the
  second summand which corresponds to the base point. It remains to
  be shown that it is~$\TT$-homotopic to the identity on the first
  summand. But, the fixed point map
  \begin{displaymath}
    [S^{r\CC+1},S^{r\CC+1}]^\TT\longrightarrow[S^1,S^1]\cong\ZZ
  \end{displaymath}
  is an isomorphism sending the identity to the identity, and the
  self-map~(\ref{eq:self-map}) induced by~$A$ is clearly the identity on
  fixed points.
\end{proof}

Any trivialisation of the Dirac bundle gives rise to an equivalence
$C^\lambda\simeq_\TT\Sigma^\lambda C_+$. Given two trivialisations,
the difference is a map~$A$ as in the preceding proposition. As a
consequence, the equivalence does not depend on the choice of the
trivialisation, and we obtain from~$f$ a well-defined invariant in
$\pi^0_\TT(\Sigma^\lambda C_+)$. On balance, the element may depend on
$f$, but the group does not.

If the diffeomorphism~$f$ is isotopic to another diffeomorphism~$g$,
there is an isotopy from the identity to~$gf^{-1}$, which can be used
to define a diffeomorphism from the mapping torus of~$f$ to the
mapping torus of~$g$ over the circle. Therefore, the family
invariants of the two mapping tori agree.

\begin{theorem}\label{thm:vanishing}
  If~$f$ is isotopic to~$g$, then the invariants of them agree.
\end{theorem}

The invariant of the identity need not be trivial unless the ordinary
invariant of~$(X,\sigma_X)$ is trivial. This deficiency can easily be
remedied as follows. For the identity, the family is trivial, and the
characteristic class is the pullback of the ordinary Bauer--Furuta
invariant. In general, the family invariant of the mapping torus
of~$f$ restricts to the same class in~$\pi^0_\TT(S^\lambda)$ as does
the pullback: the images are just the ordinary invariants of the
fibre~$(X,\sigma_X)$. Therefore, their difference is an element in the
kernel of the restriction. This kernel is identified with the 0-th
cohomotopy group of the sphere~$\Sigma^\lambda C$ by the remarks
above. This difference will be referred to as the {\it reduced
  invariant} of~$f$.

\begin{corollary}
  If~$f$ is a diffeomorphism isotopic to the identity, then the
  reduced invariant of~$f$ is zero.
\end{corollary}

The reduced invariants also behave well under iterations. Given a
diffeomorphism~$f$ and an integer~$m$, it is common to consider the
mapping tori~$Y(m)$ of the iterate~$f^m$ of~$f$ as well. This is the
total space of a family over a circle~$C(m)$. If~$m$ is a multiple
of~$n$, the family~$Y(m)$ is the pullback of~$Y(n)$ along the
map~$C(m)\rightarrow C(n)$ of degree~$m/n$. Thus, the reduced
invariant of~$f^m$ is the image of the invariant of~$f^n$ under the
map induced by~$C(m)\rightarrow C(n)$. In particular, taking~$n=1$,
this shows that the reduced invariant of~$f^m$ is divisible by~$m$.


\section{Group actions}\label{sec:diss}

If a compact Lie group~$G$ acts on~$X$ preserving a complex
spin structure~$\sigma_X$, there is an extension~$\GG$
of~$G$ by~$\TT$ such that the
homomorphism~\hbox{$G\rightarrow\Diff(X)$} lifts to a
homomorphism~\hbox{$\GG\rightarrow\Spiffc(X,\sigma_X)$}.
In~\cite{Szymik:Diss} there has been constructed an
equivariant invariant which lives in~$\pi^0_\GG(S^\lambda)$
and maps to the Bauer--Furuta invariant under the forgetful
map~\hbox{$\pi^0_\GG(S^\lambda)\rightarrow\pi^0_\TT(S^\lambda)$},
and I will now explain how this relates to the present
construction.

Given a lift as above, the universal characteristic class can be
restricted from the Thom spectrum~$\BSpiffc(X,\sigma_X)^\lambda$ to
the Thom spectrum~$\B\GG^\lambda$ to give a class
in the group~\hbox{$\pi^0_\TT(\B\GG^\lambda)$}. I do not see a direct way to
compare this with the equivariant class
in~$\pi^0_\GG(S^\lambda)$. There is, however, a universal equivariant
class in~\hbox{$\pi^0_\GG(\BSpiffc(X,\sigma_X)^\lambda)$} which maps
to both of them with the obvious maps
\begin{center}
  \mbox{ \xymatrix@C=-20pt{ &
    \pi^0_\GG(\BSpiffc(X,\sigma_X)^\lambda)\ar[dl]\ar[dr] &
    \\ \pi^0_\GG(S^\lambda) & &
    \pi^0_\TT(\BSpiffc(X,\sigma_X)^\lambda)\ar[d]\\
    {\phantom{\pi^0_\TT(\BSpiffc(X,\sigma_X)^\lambda)}} & &
    \pi^0_\TT(\B\GG^\lambda).  } }
\end{center}
No further ideas are needed for that. In this sense the present paper
conceptionally complements~\cite{Szymik:Diss}.


\section*{Acknowledgment}

I would like to thank Stefan Bauer for discussing various aspects of
this work with me.



\vfill

\parbox{\linewidth}{%
Markus Szymik\\
Department of Mathematical Sciences\\
NTNU Norwegian University of Science and Technology\\
7491 Trondheim\\
NORWAY\\
\href{mailto:markus.szymik@ntnu.no}{markus.szymik@ntnu.no}\\
\href{https://folk.ntnu.no/markussz}{folk.ntnu.no/markussz}}

\end{document}